\documentclass[10pt,reqno]{amsart}
\usepackage[utf8]{inputenc}
\usepackage[english]{babel}
\usepackage{amsmath, amssymb, amsthm}
\usepackage{mathrsfs}
\usepackage{enumitem}
\usepackage{graphicx}
\usepackage{tikz-cd}
\usepackage{xcolor}
\numberwithin{equation}{section} 

\usepackage{diagbox} 
\usepackage{tensor}  

\usepackage{hyperref} 

\usepackage[alphabetic]{amsrefs} 

\theoremstyle{plain}
\newtheorem{thm}{Theorem}[section]
\newtheorem{prop}[thm]{Proposition}
\newtheorem{lemma}[thm]{Lemma}

\newtheorem{quest}{Question}
\newtheorem{cor}[thm]{Corollary}

\theoremstyle{definition}
\newtheorem{defin}[thm]{Definition}

\newtheorem{rmk}[thm]{Remark}

\allowdisplaybreaks

\newcommand{\R}{\mathbb{R}}

\newcommand{\eps}{\varepsilon}

\DeclareMathOperator{\tr}{tr}

\def\XXint#1#2#3{{\setbox0=\hbox{$#1{#2#3}{\int}$ }
\vcenter{\hbox{$#2#3$ }}\kern-.6\wd0}}

\usepackage{scalerel,stackengine}
\stackMath
\newcommand\hhat[1]{%
\savestack{\tmpbox}{\stretchto{%
  \scaleto{%
    \scalerel*[\widthof{\ensuremath{#1}}]{\kern.1pt\mathchar"0362\kern.1pt}%
    {\rule{0ex}{\textheight}}
  }{\textheight}%
}{2.4ex}}%
\stackon[-6.9pt]{#1}{\tmpbox}%
}

\title[A Calder\'on's problem for harmonic maps]{A Calder\'on's problem for harmonic maps}
\author{Sebastián Muñoz-Thon}

\address{Department of Mathematics, Purdue University, West Lafayette, IN 47907.}
\email{smunozth@purdue.edu}

\begin{document}

\begin{abstract}
We study a version of Calderón's problem for harmonic maps between Riemannian manifolds. By using the higher linearization method, we first show that the Dirichlet-to-Neumann map determines the metric on the domain up to a natural gauge in three cases: on surfaces, on analytic manifolds, and in conformally transversally anisotropic manifolds on a fixed conformal class with injective ray transform on the transversal manifold. Next, using higher linearizations we obtain integral identities that allows us to show that the metrics on the target have the same jets at one point. In particular, if the target is analytic, the metrics are equal. We also prove an energy rigidity result, in the sense that the Dirichlet energies of harmonic maps determines the Dirichlet-to-Neumann map. 
\end{abstract}

\maketitle

\section{Introduction}

Calderón's problem \cite{calderon} asks whether one can determine the electrical conductivity of a medium by making voltage and current measurements at the boundary of the medium. This inverse method is known as Electrical Impedance Tomography (EIT). We refer to the survey \cite{gunther14} and the references therein. The generalization to Riemannian manifolds is called the \emph{anisotropic Calder\'on problem}. One considers a Riemannian manifold with boundary $(M,g)$, and from the Dirichlet-to-Neumann map $\Lambda_{g} \colon f \mapsto \partial_{\eta}u|_{\partial M}$, where $u$ solves 
\[ \Delta_{g}u=0 \text{ in }M, \qquad u|_{\partial M}=f, \]
one tries to recover the Riemannian metric $g$. Here $\Delta_{g}$ denotes the Laplace--Beltrami operator associated to $g$, and $\partial_{\eta}$ the normal derivative on $(M,g)$.

In this work we consider a Calder\'on's like problem for harmonic maps. Given $(M,g)$ and $(N,h)$ Riemannian manifolds with boundary, a map $u \in C^{\infty}(M,N)$ is \emph{harmonic} if and only if it satisfies the following systems of partial differential equations (PDE)
\begin{equation} \label{eq:hm}
    \Delta_g u^i+g^{\alpha \beta} \Gamma_{j k}^i(u) \frac{\partial u^j}{\partial x_\alpha} \frac{\partial u^k}{\partial x_\beta}=0 \quad 1 \leq i \leq \dim N.
\end{equation}
Here $\Gamma$ are the Christoffel's symbols of $(N,h)$. In invariant terms, this just means that the tension field of $u$ vanishes, that is, $\tr_{g}\nabla du=0$, where $\nabla$ denotes the connection on $T^{*}M \otimes u^{*}TN$. We deal with the following problem.

\begin{quest} \label{quest:1}
    To what extent, the Dirichlet to Neumann map
    \begin{equation} \label{eq:dn}
        \Lambda(u|_{\partial M}):=\partial_{\eta}u|_{\partial M},
    \end{equation}
    where $u$ is a solution of \eqref{eq:hm} determine $g$ and $h$?
\end{quest}

Here the Dirichlet-to-Neumann map is for solutions $u \colon (M,g) \to (N,h)$ of the harmonic map system of PDE. To have a well-defined map, we have to impose geometric conditions of $N$, and we have to choose a particular solution to \eqref{eq:hm}. We explain the details in Remark \ref{rmk:comm_main} and in Section \ref{sec:hol}.

For the domain, one should expect the usual gauge invariance, that is, that \eqref{eq:dn} determines the metric up to a diffeomorphism fixing the boundary in dimension 3 or higher, and up to a diffeomorphism and conformal class in surfaces. Regarding the target, we see that \eqref{eq:hm} presents the following gauge invariance: any metric with the same Levi-Civita connection as $h$ would give the same equation. We point out that this last condition does not imply that the metrics are homothetic, e.g., take  $(N_{j},h_{j})$ two compact Riemannian manifolds, and consider their Riemannian product $N:=(N_{1} \times N_{2},h_{1}\oplus h_{2})$. Then, for any positive constants $c_{1},c_{2}$, the family of manifolds $(N_{1} \times N_{2},c_{1}h_{1}\oplus c_{2}h_{2})$ have the same Christoffel's symbols as $N$ (but their metrics are not homothetic). However, in some cases we obtain a positive answer, for example, if the metrics are in the same conformal class, or if they are Einstein, or if its (Riemannian) holonomy group is irreducible (for example, in negative curvature). 

We obtain the following result partial answer to Question \ref{quest:1}. 

\begin{thm} \label{thm:main}
Let $M$ be a smooth compact manifold with smooth boundary, endowed with the Riemannian metrics $g,\hat{g}$. Let $\Omega \subset \R^{n}$ be a domain with Riemannian metrics $h,\hat{h}$ so that the sectional curvature of $N$ is non-positive, and that $\partial \Omega$ is strictly convex, both conditions with respect to $h$ and $h$. Let $r>0$. Assume that
\begin{equation} \label{eq:dir_1}
    \Lambda_{g,h}=\Lambda_{\hat{g},\hat{h}}  \qquad \text{on } C^{\infty}(\partial M,B_{r}(q)),
\end{equation}
where $q \in N$ and $B_{r}(q)$ denotes a ball on $\R^{n}$ with the Euclidean metric. Then we have the following:
\begin{enumerate}
    \item If $\dim M=2$, then there exists a diffeomorphism $F\colon M \to M$ such that $F|_{\partial M}=Id_{\partial M}$, and a positive function $c \in C^{\infty}(M)$ with $c|_{\partial M} \equiv 1$, such that
    \begin{equation} \label{eq:rel_metrics_2}
       F^{*}\hat{g}=cg. 
    \end{equation}
    \item If $\dim M \geq 3$, $(M,g)$ and $(M,\hat{g})$ are real-analytic up to the boundary, then there exists $F$ as before with 
    \begin{equation} \label{eq:rel_metrics_an}
       F^{*}\hat{g}=g. 
    \end{equation}
    \item If $\dim M \geq 3$, $(M,g)$ and $(M,\hat{g})$ are CTA manifolds in the same conformal class, with injective geodesic ray transform on the transversal manifold, then 
    \begin{equation} \label{eq:rel_metrics_cta}
       \hat{g}=g. 
    \end{equation}
\end{enumerate}
Furthermore, in either of the above cases, if $h(q)=\hat{h}(q)$, then the $C^{\infty}$ jets of $h$ and $\hat{h}$ coincide at $q$. 
\end{thm}

\begin{rmk} \label{rmk:comm_main} \hfill
\begin{enumerate}
    \item The hypothesis in the convexity of $\partial \Omega$ and the curvature of $\Omega$ are used for the well-posedness of the Dirichlet problem for \eqref{eq:hm}, see Theorem \ref{thm:wp}.
    \item On the third point of Theorem \ref{thm:main}, CTA stands for conformally transversally anisotropic manifold, see Definition \ref{defin:cta}. The definition of injectivity of the ray transform can be found in Definition \ref{defin:ray_inj}.
    \item We consider the target to be a domain in $\R^{n}$ so that we can apply the higher order linearization method.
    \item We do not expect to recover the conformal factor, since \eqref{eq:hm} is conformally invariant in dimension $2$.
\end{enumerate}
\end{rmk}

The proof Theorem \ref{thm:main} is based on the \emph{higher order linearization method} (HOL) introduced in \cite{klu18}. We linearize \eqref{eq:hm} around the constant map $M \ni x\mapsto q \in N$. The (first) linearization decouples the system into $\dim N$ linear PDE which says that each component of the linearized solution is a harmonic function on $M$. Therefore, we are in the setting of \cites{lu, dsfkls16}, which gives the relation between $g$ and $\hat{g}$. To obtain jets of the metrics on the codomain, we compute the second linearization, third, and more generally, the $N$th linearization of \eqref{eq:hm}. From this, we derive integral identities involving the derivatives of the Dirichlet-to-Neumann map of the corresponding order, and the Christoffel's symbols together with their derivatives. Using the relations between the metrics in $M$, we are able to obtain other integral identities, which, after choosing appropriate boundary values for the linearizations, allow us to obtain the components of the jets of the metrics on the codomain.

By the Unique Continuation Principle for real-analytic functions, Theorem \ref{thm:main} implies the following.

\begin{cor}
    In the setting of Theorem \ref{thm:main}, if $\Omega$, $h$, and $\hat{h}$ are real-analytic, they are equal.
\end{cor}

If we can extend $(\Omega,h)$, $(\Omega,\hat{h})$ in a proper way, we obtain the following consequence:

\begin{cor}
In the setting of Theorem \ref{thm:main}, assume either that 
\begin{enumerate}
    \item $\hat{h}=\tilde{c}h$, for some $\tilde{c} \in C^{\infty}(\Omega)$;
    \item $F^{*}\hat{h}=h$, for some diffeomorphism of $\Omega$ fixing $\partial M$.
\end{enumerate}
If we can extend $\Omega$ to $\mho$, and $h$, $\hat{h}$ to $H$, $\hat{H}$, respectively, so that $(\mho,H)$ and $(\mho,\hat{H})$ satisfy the hypothesis of Theorem \ref{thm:main}, then the $C^{\infty}$ jets of $h$ and $\hat{h}$ at the boundary coincide. In particular, in the analytic setting $h=\hat{h}$.
\end{cor}

For example, the extension procedure can be done if we assume the geometric conditions on $\Omega$ to be strict, that is, if $\partial \Omega$ is strictly convex and the curvature is negative.

A more geometric inverse problem is the \emph{boundary rigidity problem}. The problem asks to what extent the boundary distance function of a compact Riemannian manifold determines the metric \cite{michel}. In the 2-dimensional case, the problem was solved for \emph{simple geometries} \cite{pu2005}, and in higher dimensions under a \emph{foliation condition} \cite{suv}. A generic result was proved in \cite{su}. Some of these results have been generalized for magnetic systems \cite{dpsu} and for $\mathcal{MP}$-systems \cites{az, mt1, mt2}. One interesting generalization of boundary rigidity is considering generalization of geodesics, i.e., submanifolds of higher dimension. Recently, in  \cites{ABN20, cllo24, cllz, clz} similar problems are addressed using minimal surfaces, see below. In this setting, and since geodesics and minimal isometric immersions are a particular case of harmonic maps \cite{Jost17}*{Section 5.4}, we ask the following question.

\begin{quest} \label{quest:2}
Do the energies of harmonic maps $u \colon M \to \Omega$ determine the metrics $g$ and $h$?
\end{quest}

The energy is the classical Dirichlet energy for maps between manifolds, see \eqref{eq:energy} and Lemma \ref{lemma:e_to_dn} below. We obtain the following result about the metrics on the domain.

\begin{cor} (Energy Rigidity) \label{cor:er}
Let $M$ be smooth compact manifold with smooth boundary and let $\Omega \subset \R^{n}$ be a domain. Let $g$ and $h$ be Riemannian metrics on $M$ and $\Omega$, respectively. Assume that the sectional curvatures of $h$ are non-positive, and that $\partial \Omega$ is strictly convex with respect to $h$. Then the knowledge of $\partial M$, $g|_{\partial M}$, $h|_{B_{r}(q)}$ and the energies of harmonic maps with boundary values $\psi$, for all $\psi \in C^{\infty}(\partial M,B_{r}(q))$, determine:
\begin{enumerate}
    \item $g$ up to isometry and conformal class, if $\dim M=2$;
    \item $g$ up to isometry if $\dim M \geq 3$ and we work in the analytic setting;
    \item $g$, if $\dim M \geq 3$ is a CTA manifold in a fixed conformal class with injective ray transform on its transversal manifold.
\end{enumerate}
\end{cor}

The energy rigidity follows from Theorem \ref{thm:main} and Lemma \ref{lemma:e_to_dn}, where we show that the energies of harmonic maps determine the Dirichlet-to-Neumann map.

\subsection*{Previous literature}

Harmonic maps arise as minimizers of the Dirichlet energy for maps $u \in C^{\infty}(M,N)$
\begin{equation}
    \label{eq:energy}
    \frac{1}{2}\int_{M} \|du\|_{T^{*}M \otimes u^{*}TN}^{2} dV_{g},
\end{equation} 
where $\|\bullet\|_{T^{*}M \otimes u^{*}TN}$ denotes the norm on $T^{*}M \otimes u^{*}TN$. The first global existence result is due to Eells and Sampson \cite{ES64}, where they introduced the heat flow method, in which one deforms a map via a heat equation, so that in the limit ones obtain an harmonic map. The result works under the assumption of that the sectional curvature of the target is non-positive. This was adapted to manifolds with boundary in \cite{hamilton}, where the author also assumes convexity of the boundary of the target manifold. Harmonic maps have been playing a fundamental role in the development of geometric analysis in the last decades. From trying to extending the existence results to non-compact complete manifolds, to trying to understand the singularities that can arise in the heat flow if we remove the condition about the curvature in the target. We recommend \cites{ES95, HW08, LW08} to the interested reader. 

Regarding the anisotropic Calder\'on's problem, it has been fully solved in dimension 2, and in the real-analytic setting in the higher dimensional case \cite{lu}. It has also been solved for Einstein manifolds \cite{GSB09}. We also mention that it has been solved in some CTA geometries in some cases \cites{dksu,dsfkls16, Cekic17, kls22}. In recent years, the fractional version of the problem has been extensively studied, see \cites{grsu, fgku, cgru, ruland, fku24}, and the recent survey \cite{Covi24_survey}. There are results for other fractional elliptic operators, see \cites{Chien23, QU24}.

In recent years, the inverse problems community has been addressing questions similar to \ref{quest:1} and \ref{quest:2}, particularly regarding minimal surfaces. We now proceed to discuss some results, which served as motivation for this work.

To the author's knowledge, the first work in this direction is \cite{ABN20}, where they determine, under some geometric assumptions, the metric (up to a diffeomorphism fixing the boundary) of a 3-ball using the knowledge of least areas circumscribed by simple closed curves on $\partial M$. This kind of problem, as we mentioned before, is motivated by a higher dimensional generalization of the boundary rigidity problem, but also by its connection with physics via the AdS/CFT correspondence. 

There are some results dealing with Calder\'on's problem for minimal surfaces as well. For instance, in \cite{nurminen23} is studied the inverse problem of determining the metric of $(\R^{n},g)$, where $g$ is conformal to the Euclidean metric, via the Dirichlet-to-Neumann map of minimal surfaces. Afterwards, the same author extended their results to \emph{admissible geometries} \cite{nurminen24}, i.e., CTA manifolds so that the transversal manifold is simple. The inverse problem for minimal surfaces on manifolds of the form $\R \times M$, where $M$ is a Riemannian surface was studied in \cite{cllo24}, where they also showed that the areas of minimal surfaces determine the DN map for the minimal surface equation (as in Corollary \ref{cor:er}). In \cites{cllz, clz}, the authors proved two main results. In first place, they showed Dirichlet-to-Neumann map for the Schr\"odinger equation determines the metric and the potential (generalizing a previous result \cite{gt11}). Secondly, they proved that the Dirichlet-to-Neumann map for the minimal surfaces determine the metric and the second fundamental form of the immersion (up to some gauge in each case). As a consequence of solving the Calder\'on problem for minimal surfaces, the authors then show that the volume of minimal surfaces determine the geometry of the domain and the second fundamental form of the immersion (again, up to some gauge), similarly to Corollary \ref{cor:er}.

The results mentioned in the previous paragraph are based on the higher order linearization, the construction of complex geometrical optics (CGO), and applications of results for Calder\'on's problem for Schr\"odinger equations. For example, \cites{nurminen23, nurminen 24} use higher order linearization to obtain information about the jets of the metrics (as we do for Theorem \ref{thm:main}), 
As we mentioned before, our methods are similar. We take advantage of the nonlinearity involving derivatives of the maps and the Christoffel's symbols of the codomain to obtain information of both, the metric on the target and the metric on the codomain. However, we do not need to construct CGO since the quantities that we determine are localize at just one point, and hence we only need to show that some integrals are non-negative, instead of study their asymptotic behavior, see Section \ref{sec:pro}.

Finally, we would like to mention that there is a work about another geometric PDE: the conformal Laplacian \cite{lls}. There, adapting the methods from \cite{lu}, the authors show that in dimension greater than 3, the Dirichlet-to-Neumann map determines the metric up a diffeomorphism and conformal factor, in the analytic category. We also point out that harmonic maps already appeared in the context of inverse problems in the past. Indeed, in \cites{SU91} the authors relate the anisotropic Calderón's problem to the existence of harmonic maps which are diffeomorphism and fix the boundary.

\subsection*{Organization of the paper} In Section \ref{sec:hol} we compute the first, second, and higher order linearizations of \eqref{eq:hm} around $q \in \Omega$. In Section \ref{sec:in_id} we use these equations to obtain integral identities. In Section \ref{sec:pro} we prove Theorem \ref{thm:main}: in Subsection \ref{sub:metric_domain} we determine the metric on the domain, in subsections \ref{sub:first_der} and \ref{sub:second_der} we determine first and second derivatives of the codomain's metric at $q$, and in subsection \ref{sub:jet} we determine its jet at $q$. Finally, in Section \ref{sec:er} we prove Corollary \ref{cor:er}. 

\subsection*{Acknowledgments} The author would like to thank G. Uhlmann for suggesting the problem. The author would like to thank to Gabriel P. Paternain, P. Stefanov and G. Uhlmann also, for helpful discussions and commentaries on preliminary versions of this article. In addition, the author would like to thank University of Washington for its hospitality during his visit. The author was partly supported by NSF Grant DMS-2154489 and Ross-Lynn scholar grant.

\section{High Order Linearization}
\label{sec:hol}

Let $\psi \in C^{\infty}(\partial M,\Omega)$. Consider the Dirichlet problem for harmonic maps, that is,
\begin{equation} \label{eq:Dir_prob}
    \begin{cases}
    \Delta_g u^i+g^{\alpha \beta} \Gamma_{j k}^i(u) \frac{\partial u^j}{\partial x^{\alpha}} \frac{\partial u^k}{\partial x^{\beta}}=0 & \text{in } M, \\
    u=\psi & \text{on } \partial M.
\end{cases}
\end{equation}

Let $C_{\psi}^{\infty}(M,\Omega)$ be the closed subspace of $C^{\infty}(M,\Omega)$ consisting of smooth maps $u \colon M \to \Omega$ such that $u|_{\partial M}=\psi$. Well-possedness for \eqref{eq:Dir_prob} is given by a combination of results from \cite{hamilton} and \cite{hartman}:

\begin{thm}[\cite{ES95}*{p. 59-60}] \label{thm:wp}
Let $(M,g)$ and $(\Omega,h)$ be compact manifolds with boundary. Assume that the sectional curvature of $N$ is non-positive and that $\partial \Omega$ is convex. Then the problem \eqref{eq:Dir_prob} has a unique smooth solution in each relative homotopy class of $C_{\psi}^{\infty}(M,\Omega)$. Furthermore, $u$ depends smoothly on its boundary value.
\end{thm}

Here, a relative homotopy class $C_{\psi}^{\infty}(M,\Omega)$ is just a connected component of $C_{\psi}^{\infty}(M,\Omega)$.

As we mentioned in the introduction, we will use the high order linearization technique. Let us fix $q \in \Omega$. Note that the constant map 
\begin{align*}
    \kappa_{q} \colon M \to \Omega, \\
    x \mapsto q,
\end{align*}
is the only solution to \eqref{eq:Dir_prob} with boundary value $\partial M \ni x \mapsto q$. We will linearize \eqref{eq:Dir_prob} around $\kappa_{q}$. Let
\begin{equation} \label{eq:boundary_value}
    f_{\eps}=\sum_{j=1}^{N_{0}+1}\eps_{j}f_{j}.
\end{equation}
where $f_{j}\colon \partial M \to \R^{n}$ are smooth and $\eps=(\eps_{1},\eps_{2},\ldots,\eps_{N_{0}+1}) \in \R^{N_{0}+1}$ is small enough so that $q+f_{\eps} \in B_{r}(q)$. Here $N_{0}$ will vary depending on how many linearizations we will need. For example, for the second linearization we only need $N_{0}+1=2$, but to obtain information about the metric on the codomain using the second linearization, we will need $N_{0}+1=3$. For each $\eps$, there is a solution $u_{\eps}$ to \eqref{eq:Dir_prob}, with boundary value $f_{\eps}$, and homotopic to it. Observe that $u_{0}=\kappa_{q}$. 

To avoid overloading the notation, we will write $u$ instead of $u_{\eps}$. Following the works \cites{cllz, clz}, we introduce the notation
\begin{equation} \label{eq:linearizations}
    v_{\mu}=\partial_{\eps_{\mu}}u|_{\eps=0}, \quad v_{\mu \nu}=\partial_{\eps_{\mu}\eps_{\nu}}u|_{\eps=0},
\end{equation}
and so on. We will reserve the indices $\mu,\nu,\theta,\varphi, \mu_{j}$ for the linearizations. In order to simplify the notation, we write $\partial_{\alpha}$ to denote $\partial/\partial x^{\alpha}$. The Greek indices $\alpha,\beta$ will be used to denote quantities and derivatives arising from $M$, while the indices $i,j,k,\ell$ will be used to denote quantities and derivatives in $\Omega$. 

Before linearizing, note that since the Laplace--Beltrami operator is linear, we only have to focus on linearizing the second term on the left hand-side of \ref{eq:hm}. Indeed, is enough linearzing this term without $g^{-1}$. Therefore, we compute
\begin{align*}
    \partial_{\eps_{\mu}}(\Gamma_{j k}^i(u) \partial_{\alpha} u^{j}\partial_{\beta}u^{k}) =&(\partial_{\eps_{\mu}}\Gamma_{j k}^i(u))\partial_{\alpha} u^{j}\partial_{\beta}u^{k}+\Gamma_{ij}^{k}(u)\partial_{\alpha} \partial_{\eps_{\mu}}u^{j}\partial_{\beta}u^{k} \\
    &+\Gamma_{ij}^{k}(u)\partial_{\alpha} u^{j}\partial_{\beta}\partial_{\eps_{\mu}}u^{k}.
\end{align*}
Since $u|_{\eps=0}=\kappa_{q}$, after evaluating at $\eps=0$ all the terms vanishes. Hence, the first linearization is just 
\[ \Delta_g v_{\mu}^i=0.\]
Furthermore, observe that
\[ \partial_{\eps_{\mu}}|_{\eps=0}f_{\eps}=f_{\eps_{\mu}}. \]

To compute the second linearization, as before, we first compute the second derivative of the term involving the Christoffel's symbols. We obtain 
\begin{align*}
    \partial_{\eps_{\mu}\eps_{\nu}}(\Gamma_{jk}^{i}(u)\partial_{\alpha}u^{i}\partial_{\beta}u^{k})=&\partial_{\eps_{\nu}}( (\partial_{\eps_{\mu}} \Gamma_{jk}^{i}(u))\partial_{\alpha}u^{i}\partial_{\beta}u^{k} \\
    &+\Gamma_{jk}^{i}(u)( \partial_{\alpha}(\partial_{\eps_{\mu}}u^{i})\partial_{\beta}u^{k}+\partial_{\alpha}u^{i}\partial_{\beta}(\partial_{\eps_{\mu}}u^{i}) )  ) \\
    =&(\partial_{\eps_{\mu}\eps_{\nu}}\Gamma_{jk}^{i}(u))\partial_{\alpha}u^{i}\partial_{\beta}u^{k}  \\
    &+(\partial_{\eps_{\mu}} \Gamma_{jk}^{i}(u) )((\partial_{\alpha} \partial_{\eps_{\nu}}u^{j})\partial_{\beta}u^{k}+\partial_{\alpha}u^{j}\partial_{\beta}\partial_{\eps_{\nu}}u^{k} ) \\
    &+(\partial_{\eps_{\nu}}\Gamma_{jk}^{i}(u))( (\partial_{\alpha}\partial_{\eps_{\mu}}u^{j} )\partial_{\beta}u^{k}+\partial_{\alpha}u^{j}\partial_{\beta}\partial_{\eps_{\mu}}u^{k}  ) \\
    &+\Gamma_{jk}^{i}(u)( (\partial_{\alpha}\partial_{\eps_{\mu} \eps_{\nu}}u^{j})\partial_{\beta}u^{k}+(\partial_{\alpha}\partial_{\eps_{\mu}}u^{j})\partial_{\beta}\partial_{\eps_{\nu}}u^{k} )  \\
    &+\Gamma_{jk}^{i}(u)( (\partial_{\alpha}\partial_{\eps_{\nu}}u^{j})\partial_{\beta}\partial_{\eps_{\mu}}u^{k}+(\partial_{\alpha}u^{j}) \partial_{\beta}\partial_{\eps_{\mu}\eps_{\nu}}u^{k} ).
\end{align*}
After evaluating at $\eps=0$, we obtain 
\[ \Gamma_{jk}^{i}(q) ( \partial_{\alpha}v_{\mu}^{j}\partial_{\beta}v_{\nu}^{k}+\partial_{\alpha}v_{\nu}^{j}\partial_{\beta}u^{k} ). \]
If we multiply this expression by $g^{\alpha \beta}$ and we sum over the repeated indices, we conclude that the second linearization takes the form
\[ \Delta_g v_{\mu \nu}^i+2g^{\alpha \beta}\Gamma_{jk}^{i}(q)\partial_{\alpha}v_{\mu}^{j}\partial_{\beta}v_{\nu}^{k}=0. \]
Note that
\[ v_{\mu \nu}^{i}|_{\partial M}=\partial_{\eps_{\nu}\eps_{\mu}}f_{\eps}|_{\eps=0}=0. \]
and the same holds for higher order linearizations because we are taking linear perturbations. To obtain the third and higher linearizations, we will use the \cite{Hardy06}*{Proposition 5} to write
\begin{equation} \label{eq:high_der}
    \frac{\partial^{N}}{\partial \eps_{1}\cdots \partial \eps_{N}}(\Gamma_{jk}^{i}\partial_{\alpha}u^{j}\partial_{\beta}u^{k})=\sum_{S \in 2^{[N]}} \frac{\partial^{|S|}\Gamma_{jk}^{i}}{\prod_{j \in S}\partial\eps_{j}} \sum_{T \in 2^{[N] \setminus S} } \frac{ \partial^{|T|}\partial_{\alpha}u^{j} }{\prod_{j \in T} \partial \eps_{j} } \frac{\partial^{N-|S|-|T|} \partial_{\beta}u^{k} }{\prod_{j \notin T} \partial \eps_{j} }
\end{equation}
where $[N]=\{1,2,\ldots,N\}$ and all the derivatives with respect to $\eps_{j}$'s are evaluated at $\eps=0$. To simplify the notation, we will write
\[ \frac{\partial^{|S|}\Gamma_{jk}^{i}}{\prod_{j \in S}\partial\eps_{j}}\bigg|_{\eps=0}=\partial_{S}\Gamma_{jk}^{i}, \quad \frac{ \partial^{|T|}\partial_{\alpha}u^{j} }{\prod_{j \in T} \partial \eps_{j} } \bigg|_{\eps=0}=\partial_{\alpha}v^{j}_{T}, \quad \frac{\partial^{N-|S|-|T|}\partial_{\beta}u^{k} }{\prod_{j \notin T} \partial \eps_{j} }\bigg|_{\eps=0}=\partial_{\beta}v^{k}_{T^{c}}. \]
We observe that since $\partial_{\alpha}u_{j}|_{\eps=0}=\partial_{\beta}u|_{\eps=0}=0$, we can write equation \eqref{eq:high_der} as
\begin{equation}
    \sum_{\substack{S \in 2^{[N]}\\ 0 \leq |S| \leq N-2}} \partial_{S}\Gamma_{jk}^{i} \sum_{\substack{T \in 2^{[N] \setminus S}\\ |T| \geq 1 \\ T \sqcup S \neq [N]}} \partial_{\alpha}v_{T}^{j}  \partial_{\beta}v_{T^{c}}^{k}.
\end{equation}

We summarize the formulas of this section in the following lemma, where we also add the formula for the third linearization that will be used explicitly on sections \ref{sec:in_id} and \ref{sec:pro}.

\begin{lemma} \label{lemma:linear_eq}
Let $f$ as in \eqref{eq:boundary_value}. Then, 
\begin{enumerate}
    \item The first linearization $v_{\mu}$ solves 
    \begin{equation}
        \begin{cases} \label{eq:1_lin}
        \Delta_g v_{\mu}^i=0 & \text{in } M, \\
        v_{\ell}^{i}=f_{\mu}^{i} & \text{on } \partial M.
        \end{cases}
    \end{equation}
    \item The second linearization $v_{\mu \nu}$ solves
    \begin{equation} \label{eq:2_lin}
        \begin{cases}
        \Delta_g v_{\mu \nu}^i+2g^{\alpha \beta}\Gamma_{jk}^{i}(q)\partial_{\alpha}v_{\mu}^{j}\partial_{\beta}v_{\nu}^{k}=0 & \text{in } M, \\
        v_{\mu \nu}^{i}=0 & \text{on } \partial M.
        \end{cases}
    \end{equation}
    \item The third linearization satisfies
    \begin{equation} \label{eq:3_lin}
        \begin{cases}
        \Delta_g v_{\mu \nu \theta}^i \\
        +2g^{\alpha \beta} \displaystyle \sum_{\mathrm{cyc}}[\partial_{\ell}\Gamma_{jk}^{i}(q)v_{\mu}^{\ell}\partial_{\alpha}v_{\nu}^{j}\partial_{\beta}v_{\theta}^{k}+\Gamma_{jk}^{i}(q) \partial_{\alpha}v_{\mu \nu}^{j}\partial_{\beta}v_{\theta}^{k} ]=0 & \text{in } M, \\
        v_{\mu \nu \theta}^{i}=0 & \text{on } \partial M,
        \end{cases}
    \end{equation}
    where the sum is over $(\mu,\nu,\theta)$.
    \item The $N$th linearization solves 
    \begin{equation} \label{eq:n_lin}
        \begin{cases}
                \Delta_{g}v_{1\cdots N}^{i}+g^{\alpha \beta}\displaystyle\sum_{\substack{S \in 2^{[N]}\\ 0 \leq |S| \leq N-2}} \partial_{S} \Gamma_{jk}^{i} \sum_{\substack{T \in 2^{[N] \setminus S}\\ |T| \geq 1 \\ T \sqcup S \neq [N]}} \partial_{\alpha}v_{T}^{j} \partial_{\beta}v_{T^{c}}^{k} =0 & \text{in } M, \\
            v=0 & \text{on } \partial M.
        \end{cases}
    \end{equation}
\end{enumerate}    
\end{lemma}

\section{Integral Identities} \label{sec:in_id}

In this section we prove some integral identities satisfied by the linearizations of harmonic maps. These will be used in order to obtain Alessandrini's identities in Section \ref{sec:pro}.

\begin{lemma}[Integral identity for the second linearization] \label{lemma:first_int_id}
    \[ \int_{\partial M}f_{\theta}^{\ell} h_{\ell i}(q)\partial_{\eps_{\mu}\eps_{\nu}}\Lambda^{i}(q+f)|_{\eps=0}dS_{g}=-2\int_{M}v_{\theta}^{\ell}h_{\ell i}(q)g^{\alpha \beta}\Gamma_{jk}^{i}(q)\partial_{\alpha}v_{\mu}^{i}\partial_{\beta}v_{\nu}^{i}dV_{g}. \]
\end{lemma}

\begin{proof}
We have the following chain of equalities
\begin{align*}
    \int_{\partial M}h_{\ell i}(q)f_{\theta}^{\ell}\partial_{\eta}v_{\mu \nu}^{i}dS_{g}=&\int_{M}h_{\ell i}(q)v_{\theta}^{\ell} \Delta_{g}v_{\mu \nu}^{i}dV_{g}+\int_{M}h_{\ell i}(q)g(\nabla v_{\theta}^{\ell},\nabla v_{\mu \nu}^{i})dV_{g} \\
    =&\int_{M}h_{\ell i}(q)v_{\theta}^{\ell}(\Delta_{g}v_{\mu \nu}^{i}+2g^{\alpha \beta}\Gamma_{jk}^{i}(q)\partial_{\alpha}v_{\mu}^{j}\partial_{\beta}v_{\nu}^{k})dV_{g} \\
    &-2\int_{M}h_{\ell i}(q)v_{\theta}^{\ell}g^{\alpha \beta}\Gamma_{jk}^{i}(q)\partial_{\alpha}v_{\mu}^{i}\partial_{\beta}v_{\nu}^{i}dV_{g}\\
    &+\int_{\partial M}h_{\ell i}(q)v_{\mu \nu}^{i}\partial_{\eta}v_{\theta}^{\ell}dS_{g}-\int_{M}h_{\ell i}(q)v_{\mu \nu}^{i}\Delta_{g}v_{\theta}^{\ell}dV_{g} \\
    &=-2\int_{M}h_{\ell i}(q)v_{\theta}^{\ell}g^{\alpha \beta}\Gamma_{jk}^{i}(q)\partial_{\alpha}v_{\mu}^{i}\partial_{\beta}v_{\nu}^{i}dV_{g},
\end{align*}
where in the first two equalities we integrated by parts, while in the third one we used \eqref{eq:2_lin}, that $v_{\theta}^{\ell}$ is harmonic, and that $v_{\mu \nu}^{i}$ vanishes on the boundary of $M$.
\end{proof}

\begin{lemma}[Integral identity for the third linearization] \label{lemma:second_int_id}
    \begin{align*}
        \int_{\partial M}&f_{\varphi}^{m}h_{mi}(q) \partial_{\eps_{\mu}\eps_{\nu}\eps_{\theta}}\Lambda^{i}(q+f)|_{\eps=0}dS_{g} \\
        =&-2\int_{M}v_{\varphi}^{m}h_{mi}(q)g^{\alpha \beta} \displaystyle \sum_{\mathrm{cyc}}[\partial_{\ell}\Gamma_{jk}^{i}(q)v_{\mu}^{\ell}\partial_{\alpha}v_{\nu}^{j}\partial_{\beta}v_{\theta}^{k}+\Gamma_{jk}^{i}(q) \partial_{\alpha}v_{\mu \nu}^{j}\partial_{\beta}v_{\theta}^{k} ]dV_{g},
    \end{align*}
    where the sum over $(\mu,\nu,\theta)$.
\end{lemma}

\begin{proof}
The proof is similar to the one of Lemma \ref{lemma:first_int_id}:
\begin{align*}
    \int_{\partial M}&f_{\varphi}^{m}\partial_{\eta}v_{\mu \nu \theta}^{i}dS_{g} \\
    =&\int_{M}h_{mi}(q)v_{\varphi}^{m} \Delta_{g}v_{\mu \nu \theta}^{i}dV_{g}+\int_{M}h_{mi}(q)g(\nabla v_{\varphi}^{m},\nabla v_{\mu \nu}^{i})dV_{g} \\
    =&\int_{M}h_{mi}(q)v_{\varphi}^{m} \bigg \lbrace \Delta_{g}v_{\mu \nu \theta}^{i} \\
    &+2g^{\alpha \beta} \displaystyle \sum_{\text{cyc}}[\partial_{\ell}\Gamma_{jk}^{i}(q)v_{\mu}^{\ell}\partial_{\alpha}v_{\nu}^{j}\partial_{\beta}v_{\theta}^{k}+\Gamma_{jk}^{i}(q) \partial_{\alpha}v_{\mu \nu}^{j}\partial_{\beta}v_{\theta}^{k} ] \bigg\rbrace dV_{g} \\
    &-2\int_{M}v_{\varphi}^{m}h_{mi}(q)
    g^{\alpha \beta} \displaystyle \sum_{\text{cyc}}[\partial_{\ell}\Gamma_{jk}^{i}(q)v_{\mu}^{\ell}\partial_{\alpha}v_{\nu}^{j}\partial_{\beta}v_{\theta}^{k}+\Gamma_{jk}^{i}(q) \partial_{\alpha}v_{\mu \nu}^{j}\partial_{\beta}v_{\theta}^{k} ]dV_{g} \\
    &+\int_{\partial M}h_{mi}(q)v_{\mu \nu \theta}^{i}\partial_{\eta}v_{\varphi}^{m}dS_{g}-\int_{\partial M}h_{mi}(q)v_{\mu \eta \theta}^{i}\Delta_{g}v_{\varphi}^{m}dV_{g} \\
    =&-2\int_{M}v_{\varphi}^{m}h_{mi}(q)g^{\alpha \beta} \displaystyle \sum_{\text{cyc}}[\partial_{\ell}\Gamma_{jk}^{i}(q)v_{\mu}^{\ell}\partial_{\alpha}v_{\nu}^{j}\partial_{\beta}v_{\theta}^{k}+\Gamma_{jk}^{i}(q) \partial_{\alpha}v_{\mu \nu}^{j}\partial_{\beta}v_{\theta}^{k} ]dV_{g},
\end{align*}    
where in the first two equalities we used integration by parts, and in the third one we used equation \eqref{eq:3_lin}, together with equation \eqref{eq:1_lin} and $v_{\mu \nu \theta}^{i}|_{\partial M}=0$ (from Lemma \ref{lemma:linear_eq} (2)).
\end{proof}

Finally, we have the formula for the $N$th linearization

\begin{lemma} [Integral identity for the $N$th linearization] \label{lemma:N_int_id}
    \begin{align*}
        &\int_{\partial M}f_{N+1}^{\ell}h_{\ell i}(q) 
 \frac{\partial^{N}}{\partial \eps_{1}\cdots \partial \eps_{N}}\Lambda^{i}(q+f)|_{\eps=0}dS_{g} \\
 &=-\int_{M}v_{N+1}^{\ell}h_{\ell i}(q)g^{\alpha \beta}\displaystyle\sum_{\substack{S \in 2^{[N]}\\ 0 \leq |S| \leq N-2}} \partial_{S}\Gamma_{jk}^{i} \sum_{\substack{T \in 2^{[N] \setminus S}\\ |T| \geq 1 \\ T \sqcup S \neq [N]}} \partial_{\alpha}v_{T}^{j} \partial_{\beta}v_{T^{c}}^{k} dV_{g}.
    \end{align*}
\end{lemma}

\begin{proof}
As before, we have
\begin{align*}
    \int_{\partial M} &f_{N+1}^{\ell}h_{\ell i}(q)\partial_{\eta}v_{1\cdots N}^{i}dS_{g} \\
    =&\int_{M}f_{N+1}^{\ell}h_{\ell i}(q) \Delta_{g}v_{1\cdots N}^{i}dV_{g}+\int_{M}h_{\ell i}(q)g(\nabla v_{N+1}^{\ell},\nabla v_{1 \cdots N}^{i})dV_{g} \\
    =& \int_{M}f_{N+1}^{\ell}h_{\ell i}(q) \Delta_{g}v_{1\cdots N}^{i}dV_{g} \\
    &+\int_{M}v_{N+1}^{\ell}h_{\ell i}(q)g^{\alpha \beta}\displaystyle\sum_{\substack{S \in 2^{[N]}\\ 0 \leq |S| \leq N-2}} \partial_{S}\Gamma_{jk}^{i} \sum_{\substack{T \in 2^{[N] \setminus S}\\ |T| \geq 1 \\ T \sqcup S \neq [N]}} \partial_{\alpha}v_{T}^{j} \partial_{\beta}v_{T^{c}}^{k} dV_{g} \\
    &-\int_{M}v_{N+1}^{\ell}h_{\ell i}(q)g^{\alpha \beta}\displaystyle\sum_{\substack{S \in 2^{[N]}\\ 0 \leq |S| \leq N-2}} \partial_{S}\Gamma_{jk}^{i} \sum_{\substack{T \in 2^{[N] \setminus S}\\ |T| \geq 1 \\ T \sqcup S \neq [N]}} \partial_{\alpha}v_{T}^{j}\partial_{\beta}v_{T}^{k} dV_{g} \\
    &+\int_{\partial M} h_{i \ell}(q)v_{1\cdots N}^{i} \partial_{\eta}v_{N+1}^{\ell}dS_{g}-\int_{M} h_{\ell i}(q)v_{1\cdots N}^{i} \Delta_{g}v_{N+1}^{\ell}dV_{g} \\
    =& -\int_{M}v_{N+1}^{\ell}h_{\ell i}(q) g^{\alpha \beta}\displaystyle\sum_{\substack{S \in 2^{[N]}\\ 0 \leq |S| \leq N-2}} \partial_{S}\Gamma_{jk}^{i} \sum_{\substack{T \in 2^{[N] \setminus S}\\ |T| \geq 1 \\ T \sqcup S \neq [N]}} \partial_{\alpha}v_{T}^{j} \partial_{\beta}v_{T^{c}}^{k} dV_{g},
\end{align*}
where in the first two equalities we integrated by parts; and in the final one we used the $N$th linearized equation \eqref{eq:n_lin}, together with the facts that $v_{1\cdots N}^{i}|_{\partial M}=0$ (Lemma \ref{lemma:linear_eq} (4)), and that $v_{N+1}^{\ell}$ is harmonic (Lemma \ref{lemma:linear_eq} (1)).
\end{proof}

\section{Proof of the Main Theorem} \label{sec:pro}

We divide the analysis into several steps. First, using the first linearization of \eqref{eq:hm}, we determine the metric on the domain up to gauge. Then, using the second and the third linearization, we determine the Christoffel's symbols and their first derivatives on the target. Finally, we show how we can recover the full jet of the metric at $q$ using higher order linearizations.

\subsection{Determination of the metric on the domain} \label{sub:metric_domain}

First we determine the metric on the domain. We recall the following definitions

\begin{defin} \label{defin:cta}
Let $(M, g)$ be a compact manifold with smooth boundary with $\dim M \geq 3$
\begin{enumerate}
    \item $(M, g)$ is called \emph{transversally anisotropic} if $(M, g) \subset \subset(T, g)$ where $T=\mathbb{R} \times M_0$, $g=e \oplus g_0,(\mathbb{R}, e)$ is the Euclidean line, and $\left(M_0, g_0\right)$ is some compact $(n-1)$ dimensional manifold with boundary. Here $\left(M_0, g_0\right)$ is called the transversal manifold.
    \item $(M, g)$ is called \emph{conformally transversally anisotropic} (CTA) if $(M, c g)$ is transversally anisotropic for some smooth positive function $c$.
\end{enumerate}
\end{defin}

\begin{defin} \label{defin:ray_inj}
Let $(M,g)$ be a CTA manifold. We say that the geodesic X-ray transform on the transversal manifold $(M_{0}, g_{0})$ is \emph{injective} if any function $f \in C (M_{0})$ which integrates to zero over all non-tangential geodesics in $M_0$ must satisfy $f=0$. Here, a unit speed geodesic segment $\gamma \colon [0, L] \to M_{0}$ is \emph{called non-tangential} if $\dot{\gamma}(0), \dot{\gamma}(L)$ are non-tangential vectors on $\partial M_{0}$ and $\gamma(t) \in M_0^{\mathrm{int}}$ for $0<t<L$.    
\end{defin}

For more details about CTA manifolds we refer to \cites{dksu,dsfkls16, Cekic17, kls22}.

\begin{lemma} \label{lemma:metrics_domain}
    Let $M$, $g,\hat{g}$, $\Omega$, $h$, and $\hat{h}$ be as in Theorem \ref{thm:main}. Then \eqref{eq:rel_metrics_2} and \eqref{eq:rel_metrics_an} hold on the corresponding cases.
\end{lemma}

\begin{proof}
Let $f=f_{\eps}$ be as in \eqref{eq:boundary_value}. Here we just need $N_{0}=0$, but we can take $N_{0}$ to be larger if need, as we do in the higher linearizations. Since $\Lambda_{g,h}=\Lambda_{\hat{g},\hat{h}}$, by Lemma \ref{lemma:linear_eq}, we conclude that the Dirichlet-to-Neumann maps of the systems
\[
\begin{cases}
    \Delta_{g} v_{\mu}^i=0 & \text{in } M, \\
    v_{\mu}^{i}=f_{\mu}^{i} & \text{on } \partial M,
\end{cases}  
\qquad \begin{cases}
    \Delta_{\hat{g}} \hat{v}_{\mu}^i=0 & \text{in } M, \\
     \hat{v}_{\mu}^{i}=f_{\mu}^{i} & \text{on } \partial M,
\end{cases}
\]
are the same. Here $v$ and $\hat{v}$ denote the first linearization of the solution of \eqref{eq:hm} from $(M,g)$ to $(\Omega,h)$, and from $(M,\hat{g})$ to $(\Omega,\hat{h})$, respectively. In the case when $\dim M=2$, \cite{lu}*{Theorem 1.1 (i)} implies the existence of a diffeomorphism $F \colon M \to M$ with $F|_{\partial M}=\text{Id}|_{\partial M}$, and a non-negative function $c$ so that $c|_{\partial M} \equiv 1$, satisfying \eqref{eq:rel_metrics_2}. In the real analytic setting, the same analysis holds (applying \cite{lu}*{Theorem 1.1 (ii)} this time), but without the conformal factor, that is, \eqref{eq:rel_metrics_an} holds. Furthermore, in the CTA case, we have the same conclusion without the diffeomorphism nor the conformal factor, that is, \eqref{eq:rel_metrics_cta} holds.
\end{proof}

\begin{rmk}
We point out that the same analysis would work in the higher-dimensional case (without the analytical regularity assumptions or without the CTA and X-ray assumptions) if the Calder\'ons problem would be solved in higher dimensions.
\end{rmk}

\subsection{Determination of the first derivative of the target's metric} \label{sub:first_der}

By
\begin{equation} \label{eq:first_der_h}
    \partial_{\ell}h_{ij}=h_{rj}\Gamma_{i\ell}^{r}+h_{is}\Gamma_{\ell j}^{s},
\end{equation}
we observe that to determine the first derivatives of the metric on the codomain, it is enough to determine the Christoffel's symbols. Hence, we will deal with the equality of the Christoffel's symbols on the target. First, we obtain a relation between $v$ and $\hat{v}$. Indeed, let $\tilde{v}=F^{*}\hat{v}$. If we are in the 2-dimensional case, then $\Delta_{cg}=c^{-1}\Delta_{g}$. So,
\[ \Delta_{g}\tilde{v}_{\mu}^{i}=\Delta_{c^{-1}F^{*}\hat{g}}(v_{\mu} \circ F)^{i}=cF^{*}(\Delta_{\hat{g}} \hat{v}_{\mu}^{i})=0. \]
Since $F$ is the identity in the boundary, then $v$ and $\tilde{v}$ have the same boundary values. Therefore, $v$ and $\tilde{v}$ satisfy the same boundary value problem. By uniqueness of the
Dirichlet problem of elliptic operators (\cite{GT01}*{Theorem 6.15}) we conclude
\begin{equation} \label{eq:rel_first_lin}
    v=\tilde{v}
\end{equation}
As before, the same analysis applies in the analytic case (but without $c$). In the same spirit, we also have
\begin{equation} \label{eq:rel_first_lin_CTA}
    v=\hat{v},
\end{equation}
in the CTA case. Finally, observe that we also have, 
\begin{equation} \label{eq:rel_vols_2}
    F^{*}dV_{\hat{g}}=cdV_{g}
\end{equation}
in the 2-dimensional case, while
\begin{equation} \label{eq:rel_vols_n}
    F^{*}dV_{\hat{g}}=dV_{g}
\end{equation}
on the real analytic (and higher dimensional) case, whereas 
\begin{equation} \label{eq:rel_vols_cta}
    dV_{\hat{g}}=dV_{g},
\end{equation}
in the CTA case.

With this setting, we are ready to prove the main ingredient to determine the Christoffel's symbols.

\begin{lemma}[First Alessandrini's identity] \label{lemma:first_ale} In the setting of Theorem \ref{thm:main}, we have
\begin{equation} \label{eq:ale_1_full}
    \int_{M}h_{\ell i}(q)(\Gamma_{jk}^{i}(q)-\hat{\Gamma}_{jk}^{i}(q))v_{\theta}^{\ell}g^{\alpha \beta}\partial_{\alpha}v_{\mu}^{j}\partial_{\beta}v_{\nu}^{k}dV_{g}=0,
\end{equation}
    where $\Gamma$ and $\hat{\Gamma}$ corresponds to the Christoffel's symbols of $(\Omega,h)$ and $(\Omega,\hat{h})$, respectively.
\end{lemma}

\begin{proof}
    We have 
\begin{align*}
    \int_{M}h_{\ell i}(q)\Gamma_{jk}^{i}(q)v_{\theta}^{\ell}g^{\alpha \beta}\partial_{\alpha}v_{\mu}^{j}\partial_{\beta}v_{\nu}^{k}dV_{g}&=\int_{M}h_{\ell i}(q)\hat{\Gamma}_{jk}^{i}(q)\hat{v}_{\theta}^{\ell}\hat{g}^{\alpha \beta}\partial_{\alpha}\hat{v}_{\mu}^{j}\partial_{\beta}\hat{v}_{\nu}^{k}dV_{\hat{g}} \\
    &=\int_{M} h_{\ell i}(q)\hat{\Gamma}_{jk}^{i}(q)\tilde{v}_{\theta}^{\ell}F^{*}\hat{g}^{\alpha \beta}\partial_{\alpha}\tilde{v}_{\mu}^{j}\partial_{\beta}\tilde{v}_{\nu}^{k} F^{*}dV_{\hat{g}} \\
    &=\int_{M} h_{\ell i}(q)\hat{\Gamma}_{jk}^{i}(q)\tilde{v}_{\theta}^{\ell}g^{\alpha \beta}\partial_{\alpha}\tilde{v}_{\mu}^{j}\partial_{\beta}\tilde{v}_{\nu}^{k} dV_{g} \\
    &=\int_{M} h_{\ell i}(q)\hat{\Gamma}_{jk}^{i}(q)v_{\theta}^{\ell}g^{\alpha \beta}\partial_{\alpha}v_{\mu}^{j}\partial_{\beta}v_{\nu}^{k} dV_{g},
\end{align*}
where in the first equality we used Lemma \ref{lemma:first_int_id}, that $h(q)=\hat{h}(q)$, and the fact that the second derivatives of $\Lambda_{g,h}=\Lambda_{\hat{g},\hat{h}}$ coincide and are the DN maps of the second linearized equations \eqref{eq:2_lin}; in the third equality we used \eqref{eq:rel_metrics_2} together with \eqref{eq:rel_vols_2} if $\dim M=2$, and \eqref{eq:rel_metrics_an} together with \eqref{eq:rel_vols_n} in the higher-dimensional and analytic case; and in the fourth equality, we used \eqref{eq:rel_first_lin}. The proof in the CTA case is more direct: we can go directly from the second to the fifth equality using \eqref{eq:rel_vols_cta} and \eqref{eq:rel_first_lin_CTA}.
\end{proof}

Now we determine the Christoffel's symbols.

\begin{lemma} \label{lemma:same_chris} 
In the setting of Theorem \ref{thm:main}
    $\Gamma_{jk}^{i}(q)=\hat{\Gamma}_{jk}^{i}(q)$.
\end{lemma}

\begin{proof}
Fix $j_{0},k_{0},\ell_{0} \in \{1,\ldots,n\}$. Let $a \in C^{\infty}(\partial M)$ to be non-constant, and let $C$ be a positive constant. Choose 
\[ f_{\mu}^{j}=\begin{cases}
    a & \text{if } j=j_{0}, \\
    0 & \text{in other case},
\end{cases} \quad f_{\nu}^{k}=\begin{cases}
    a & \text{if } k=j_{0}, \\
    0 & \text{in other case},
\end{cases} \quad f_{\theta}^{\ell}=\begin{cases}
    C & \text{if } \ell=\ell_{0}, \\
    0 & \text{in other case},
\end{cases} \]
Then, by the uniqueness of solutions to elliptic Dirichlet problems (\cite{GT01}*{Theorem 6.15}), we obtain that
\[ v_{\mu}^{j}=0 \text{ for } \{1,\ldots,n\} \setminus \{j_{0}\}, \quad v_{\nu}^{k}=0 \text{ for } \{1,\ldots,n\} \setminus \{k_{0}\},\]
and
\[ v_{\mu}^{j_{0}}=v_{\nu}^{k_{0}}, \qquad v_{\theta}^{\ell_{0}}=C. \]
Therefore, from Lemma \ref{lemma:first_ale}, we obtain 
\begin{equation} \label{eq:ale_1}
    Ch_{\ell_{0}i}(q)(\Gamma_{j_{0}k_{0}}^{i}(q)-\hat{\Gamma}_{j_{0}k_{0}}^{i}(q))\int_{M}g^{\alpha \beta}\partial_{\alpha}v_{\mu}^{j_{0}}\partial_{\beta}v_{\nu}^{j_{0}}dV_{g}=0.
\end{equation}
Observe that the difference between \eqref{eq:ale_1_full} and \eqref{eq:ale_1} is that in the first one, we are adding over $j$, $k$, and $\ell$, while in the second one, these values are fixed to be $j_{0}$, $k_{0}$, and $\ell_{0}$, respectively. Note that we can rewrite \eqref{eq:ale_1} as
\[ Ch_{\ell_{0}i}(q)(\Gamma_{j_{0}k_{0}}^{i}(q)-\hat{\Gamma}_{j_{0}k_{0}}^{i}(q))\int_{M}|dv_{\mu}^{j_{0}}|_{g}^{2}dV_{g}=0. \]
Now, assume without loss of generality that $h_{\ell_{0}i}(\Gamma_{j_{0}k_{0}}^{i}(q) -\hat{\Gamma}_{j_{0}k_{0}}^{i}(q)) \geq 0$. We have two cases here: 
\begin{equation} \label{eq:=chris}
    h_{\ell_{0}i}(q)(\Gamma_{j_{0}j_{0}}^{i}(q)-\hat{\Gamma}_{j_{0}j_{0}}^{i}(q))=0
\end{equation}
or $dv_{\mu}^{j_{0}}=0$. In the later case, we conclude that $v_{\mu}^{j_{0}}$ is constant, which implies that $a$ is constant as well. However, this contradicts our choice of $a$. Hence, \eqref{eq:=chris} must holds. We can repeat the whole proof for any $\ell_{0}$, which implies that $\Gamma_{jk}^{i}(q)=\hat{\Gamma}_{jk}^{i}(q)$. 
\end{proof}

We summarize our results of this section as follows.

\begin{prop} \label{prop:first_der}
    In the setting of Theorem \ref{thm:main}, we have $\partial_{\ell}h_{ij}(q)=\partial_{\ell}\hat{h}_{ij}(q)$.
\end{prop}

\subsection{Determination of the second derivative of the target's metric} \label{sub:second_der}

By taking the derivative on \eqref{eq:first_der_h}, we see that the second derivative of the metric at $q$ depends on the value of the metric, its first derivative, and the first derivative of the 
Christoffel's symbols, all valuated at $q$. By our hypothesis and the analysis of Section \ref{sub:first_der}, is enough to determine the first derivative of the Christoffel's symbols to obtain the result. As before, first we obtain a relation between $v_{\mu \nu}$ and $\hat{v}_{\mu \nu}$. Let $\tilde{v}=F^{*}\hat{v}$. If we are in the 2-dimensional case, then,
\begin{align*}
    \Delta_g \tilde{v}_{\mu \nu}^i+2g^{\alpha \beta}\Gamma_{jk}^{i}(q)\partial_{\alpha}v_{\mu}^{j}\partial_{\beta}v_{\nu}^{k} &=\Delta_{c^{-1}F^{*}\hat{g}} \tilde{v}_{\mu \nu}^i+2c(F^{*}\hat{g})^{\alpha \beta}\Gamma_{jk}^{i}(q)\partial_{\alpha}v_{\mu}^{j}\partial_{\beta}v_{\nu}^{k} \\
    &=\Delta_{c^{-1}F^{*}\hat{g}} \tilde{v}_{\mu \nu}^i+2c(F^{*}\hat{g})^{\alpha \beta}\Gamma_{jk}^{i}(q)\partial_{\alpha}\tilde{v}_{\mu}^{j}\partial_{\beta}\tilde{v}_{\nu}^{k} \\
    &=cF^{*} ( \Delta_{\hat{g}} \hat{v}_{\mu \nu}^i+2\hat{g}^{\alpha \beta}\Gamma_{jk}^{i}(q)\partial_{\alpha}\hat{v}_{\mu}^{j}\partial_{\beta}\hat{v}_{\nu}^{k} ) \\
    &=cF^{*} ( \Delta_{\hat{g}} \hat{v}_{\mu \nu}^i+2\hat{g}^{\alpha \beta}\hat{\Gamma}_{jk}^{i}(q)\partial_{\alpha}\hat{v}_{\mu}^{j}\partial_{\beta}\hat{v}_{\nu}^{k} ) \\
    &=0,
\end{align*}
were we used \eqref{eq:rel_metrics_2} on the first equality, \eqref{eq:rel_first_lin} on the second one, Lemma \ref{lemma:same_chris} on the fourth one, and Lemma \ref{lemma:linear_eq} (ii) one the last one. The same analysis works in the analytic case, by using \eqref{eq:rel_metrics_an} instead of \eqref{eq:rel_metrics_2}. Since $F$ is the identity in the boundary, then $v_{\mu \nu}$ and $\tilde{v}_{\mu \nu}$ have the same boundary values. Therefore,
\begin{equation} \label{eq:rel_second_lin}
    v_{\mu \nu}=\tilde{v}_{\mu \nu}
\end{equation}
by uniqueness of the Dirichlet problem for elliptic PDE (\cite{GT01}*{Theorem 6.15}). The same procedure works in the CTA (case since $F=Id_{M}$), and we obtain
\begin{equation} \label{eq:rel_second_lin_CTA}
    v_{\mu \nu}=\tilde{v}_{\mu \nu}
\end{equation}

We also need an identity that relates the first derivatives of the Christoffel's symbols.

\begin{lemma}[Second Alessandrini's identity] \label{lemma:second_ale} In the setting of Theorem \ref{thm:main}, we have
\[ \int_{M}v_{\varphi}^{m}h_{mi}(q)g^{\alpha \beta} (\partial_{\ell}\Gamma_{jk}^{i}(q)-\partial_{\ell}\hat{\Gamma}_{jk}^{i}(q))\sum_{\mathrm{cyc}}v_{\mu}^{\ell}\partial_{\alpha}v_{\nu}^{j}\partial_{\beta}v_{\theta}^{k}dV_{g}=0, \]
where $\Gamma$ and $\hat{\Gamma}$ corresponds to the Christoffel's symbols of $(\Omega,h)$ and $(\Omega,\hat{h})$, respectively, and the cyclic sum is over $(\mu,\nu,\theta)$.
\end{lemma}

\begin{proof}
As before, this time we obtain 
\begin{align*}
    &\int_{M}v_{\varphi}^{m}h_{mi}(q)g^{\alpha \beta} \displaystyle \sum_{\text{cyc}}[\partial_{\ell}\Gamma_{jk}^{i}(q)v_{\mu}^{\ell}\partial_{\alpha}v_{\nu}^{j}\partial_{\beta}v_{\theta}^{k}+\Gamma_{jk}^{i}(q) \partial_{\alpha}v_{\mu \nu}^{j}\partial_{\beta}v_{\theta}^{k} ]dV_{g} \\
    =&\int_{M}\hat{v}_{\varphi}^{m}h_{mi}(q)\hat{g}^{\alpha \beta} \displaystyle \sum_{\text{cyc}}[\partial_{\ell}\hat{\Gamma}_{jk}^{i}(q)\hat{v}_{\mu}^{\ell}\partial_{\alpha}\hat{v}_{\nu}^{j}\partial_{\beta}\hat{v}_{\theta}^{k}+\hat{\Gamma}_{jk}^{i}(q) \partial_{\alpha}\hat{v}_{\mu \nu}^{j}\partial_{\beta}\hat{v}_{\theta}^{k} ]dV_{\hat{g}} \\
    =&\int_{M}\tilde{v}_{\varphi}^{m}h_{mi}(q)F^{*}\hat{g}^{\alpha \beta} \displaystyle \sum_{\text{cyc}}[\partial_{\ell}\hat{\Gamma}_{jk}^{i}(q)\tilde{v}_{\mu}^{\ell}\partial_{\alpha}\tilde{v}_{\nu}^{j}\partial_{\beta}\tilde{v}_{\theta}^{k}+\hat{\Gamma}_{jk}^{i}(q) \partial_{\alpha}\tilde{v}_{\mu \nu}^{j}\partial_{\beta}\tilde{v}_{\theta}^{k} ]F^{*}dV_{\hat{g}}\\
    =&\int_{M}\tilde{v}_{\varphi}^{m}h_{mi}(q)g^{\alpha \beta} \displaystyle \sum_{\text{cyc}}[\partial_{\ell}\hat{\Gamma}_{jk}^{i}(q)\tilde{v}_{\mu}^{\ell}\partial_{\alpha}\tilde{v}_{\nu}^{j}\partial_{\beta}\tilde{v}_{\theta}^{k}+\hat{\Gamma}_{jk}^{i}(q) \partial_{\alpha}\tilde{v}_{\mu \nu}^{j}\partial_{\beta}\tilde{v}_{\theta}^{k} ]dV_{g}\\
    =&\int_{M}v_{\varphi}^{m}h_{mi}(q)g^{\alpha \beta} \displaystyle \sum_{\text{cyc}}[\partial_{\ell}\hat{\Gamma}_{jk}^{i}(q)v_{\mu}^{\ell}\partial_{\alpha}v_{\nu}^{j}\partial_{\beta}v_{\theta}^{k}+\hat{\Gamma}_{jk}^{i}(q) \partial_{\alpha}v_{\mu \nu}^{j}\partial_{\beta}v_{\theta}^{k} ]dV_{g} \\
    =&\int_{M}v_{\varphi}^{m}h_{mi}(q)g^{\alpha \beta} \displaystyle \sum_{\text{cyc}}[\partial_{\ell}\hat{\Gamma}_{jk}^{i}(q)v_{\mu}^{\ell}\partial_{\alpha}v_{\nu}^{j}\partial_{\beta}v_{\theta}^{k}+\Gamma_{jk}^{i}(q) \partial_{\alpha}v_{\mu \nu}^{j}\partial_{\beta}v_{\theta}^{k} ]dV_{g},
\end{align*}
where the first equality we used Lemma \ref{lemma:second_int_id}, that $h(q)=\hat{h}(q)$, and the fact that the third derivatives of $\Lambda_{g,h}=\Lambda_{\hat{g},\hat{h}}$ coincide and are the DN maps of the third linearized equations \eqref{eq:3_lin}; in the third equality we used \eqref{eq:rel_metrics_2} together with \eqref{eq:rel_vols_2} if $\dim M=2$, and \eqref{eq:rel_metrics_an} together with \eqref{eq:rel_vols_n} in the higher dimensional and analytic case; in the fourth we used \eqref{eq:rel_first_lin} and \eqref{eq:rel_second_lin}; and in the last equality we used Lemma \ref{lemma:same_chris}. Again, the proof in the CTA case is more direct: we can go directly from the first to the fourth equality using \eqref{eq:rel_vols_cta} and \eqref{eq:rel_first_lin_CTA}, and then apply Lemma \ref{lemma:same_chris} to obtain the last equality. Finally, canceling the repeated terms gives the required equality.
\end{proof}

\begin{lemma} \label{lemma:same_chris_der1}
In the setting of Theorem \ref{thm:main}, we have $\partial_{\ell}\Gamma_{jk}^{i}(q)=\partial_{\ell}\hat{\Gamma}_{jk}^{i}(q)$.
\end{lemma}

\begin{proof}
Fix $j_{0},k_{0},\ell_{0},m_{0} \in \{1,\ldots,n\}$. Let $a \in C^{\infty}(\partial M)$ with $a$ to be non-constant, and let $C$ be a positive constant. Choose 
\[ f_{\mu}^{\ell}=\begin{cases}
    C & \text{if } \ell=\ell_{0}, \\
    0 & \text{in other case},
\end{cases} \quad f_{\nu}^{j}=\begin{cases}
    a & \text{if } j=j_{0}, \\
    0 & \text{in other case},
\end{cases}\]
\[ f_{\theta}^{k}=\begin{cases}
    a & \text{if } k=k_{0}, \\
    0 & \text{in other case},
\end{cases} \quad f_{\varphi}^{m}=\begin{cases}
    C & \text{if } m=m_{0}, \\
    0 & \text{in other case}.
\end{cases} \]
Then, by uniqueness of solutions to elliptic Dirichlet problems (\cite{GT01}*{Theorem 6.15}), we obtain that
\[ v_{\mu}^{\ell}=v_{\nu}^{j}=v_{\theta}^{k}=v_{\varphi}^{m}=0,\]
$\ell \in \{1,\ldots,n \} \setminus \{\ell_{0}\}$, $j \in \{1,\ldots,n \} \setminus \{j_{0}\}$, $k \in \{1,\ldots,n \} \setminus \{k_{0}\}$, and $m \in \{1,\ldots,n \} \setminus \{m_{0}\}$, respectively, and
\[ v_{\varphi}^{m_{0}}=v_{\mu}^{\ell_{0}}=C, \qquad v_{\nu}^{j_{0}}=v_{\theta}^{k_{0}}. \]
Then, by Lemma \ref{lemma:second_ale}, we have
\[ C^{2}h_{m_{0}i}(q)(\partial_{\ell_{0}}\Gamma_{j_{0}k_{0}}^{i}(q)-\partial_{\ell_{0}}\hat{\Gamma}_{j_{0}k_{0}}^{i}(q)) \int_{M} |dv_{\nu}^{j_{0}}|_{g}^{2}dV_{g}=0. \]
Assume now without loss of generality that $h_{m_{0}i}(q)(\partial_{\ell_{0}}\Gamma_{j_{0}k_{0}}^{i}(q)-\partial_{\ell_{0}}\hat{\Gamma}_{j_{0}k_{0}}^{i}(q)) \geq 0$. If this equality is strict, then $dv_{\nu}^{j_{0}} \equiv 0$, which implies that $v_{\nu}^{j_{0}}$ is constant, and hence $a$ is constant as well. This contradiction implies 
\[ h_{m_{0}i}(q)(\partial_{\ell_{0}}\Gamma_{j_{0}k_{0}}^{i}(q)-\partial_{\ell_{0}}\hat{\Gamma}_{j_{0}k_{0}}^{i}(q))= 0. \]
Since $m_{0}$ was arbitrary, the same proof yields
\[ h_{mi}(q)(\partial_{\ell_{0}}\Gamma_{j_{0}k_{0}}^{i}(q)-\partial_{\ell_{0}}\hat{\Gamma}_{j_{0}k_{0}}^{i}(q))= 0 \qquad \forall m \in \{1,\ldots,n\}. \]
Then, we can invert $h$ to obtain $\partial_{\ell_{0}}\Gamma_{j_{0}k_{0}}^{i}(q)=\partial_{\ell_{0}}\hat{\Gamma}_{j_{0}k_{0}}^{i}(q)$.
Finally, since $\ell_{0},j_{0},k_{0}$ were arbitrary, we obtain the desire conclusion.    
\end{proof}

There results on this section imply:

\begin{prop} \label{prop:second_der}
In the setting of Theorem \ref{thm:main}, we have $\partial_{\ell k}h_{ij}(p)=\partial_{\ell k}\hat{h}_{ij}(p)$.
\end{prop}

\subsection{Determination of the jet of target's metric} \label{sub:jet}
In this sectionm we finish the proof of Theorem \ref{thm:main}, by determining the full jet of the metric on the codomain.

First, we point out the following fact.

\begin{lemma} \label{lemma:ind_der}
    The $N$th derivative of $h$ at $q$ depends on the $k$th derivatives of $h$ and $\Gamma_{ij}^{k}$ at $q$ only, where $k \in \{0,\ldots,N-1\}$ only.
\end{lemma}

We omit the proof, which is by induction. For example, \eqref{eq:first_der_h} is the case $N=1$.

\begin{lemma} \label{lemma:rel_N_lin}
    Let $N \geq 2$. Consider the setting of Theorem \ref{thm:main}. Assume that 
    \begin{equation} \label{eq:rel_N-1_lin}
        \begin{cases}
            v_{\mu_{1}} &=\tilde{v}_{\mu_{1}}, \\ 
       v_{\mu_{1} \mu_{2}} &=\tilde{v}_{\mu_{1} \mu_{2}}, \\
       & \vdots \\
       v_{\mu_{1} \mu_{2} \cdots \mu_{N-1}} &=\tilde{v}_{\mu_{1} \mu_{2} \cdots \mu_{N-1}}, \\
        \end{cases}
    \end{equation}
    for any $\mu_{j} \in \{1,\ldots,N-1\}$. Assume further that any derivative of the Christoffel's symbols up to order $N-2$ coincide. Then 
    \begin{equation} \label{eq:rel_N_lin}
        v_{\mu_{1} \mu_{2} \cdots \mu_{N}} =\tilde{v}_{\mu_{1} \mu_{2} \cdots \mu_{N}}.
    \end{equation}
    In the CTA case, we consider the hypothesis with $F=Id_{M}$, so that the linearizations of $u$ and $\hat{u}$ coincide up to order $N-1$, and the conclusion is that the linearizations of $u$ and $\hat{u}$ coincide up to order $N$.
\end{lemma}

\begin{proof}
    Let us write
\[ v_{S}^{A}=v_{s_{1}}^{A_{1}}\cdots v_{s_{|S|}}^{A_{|S|}}, \]
so that
\[ \frac{\partial^{|S|}\Gamma_{jk}^{i}}{\prod_{j \in S}\partial\eps_{j}}\bigg|_{\eps=0}=\partial_{S}\Gamma_{jk}^{i}=\partial_{A}\Gamma_{jk}^{i}v_{S}^{A}+\text{l.o.t}. \]
We will write $\widetilde{\text{l.o.t.}}$ to denote the pullback by $F$ of the expression in $\text{l.o.t.}$. In this case, for $|S| \leq N-2$ we have
\begin{equation} \label{eq:S=tildeS}
    \partial_{\tilde{S}}\Gamma_{jk}^{i}:=\partial_{A}\Gamma_{jk}^{i}\tilde{v}_{S}^{A}+\widetilde{\text{l.o.t}}=\partial_{A}\Gamma_{jk}^{i}\tilde{v}_{S}^{A}+\text{l.o.t}=\partial_{S}\Gamma_{jk}^{i},
\end{equation}
where first we used the equality of the derivatives of the Christoffel's symbols and \eqref{eq:rel_N-1_lin}, and \eqref{eq:rel_N-1_lin} in the last step. Using the same notation as before but with $\hat{\Gamma}$, we obtain 
\begin{equation} \label{eq:S=hatS}
    \partial_{S}\hat{\Gamma}_{jk}^{i}=\partial_{A}\hat{\Gamma}_{jk}^{i}\hat{v}_{S}^{A}+\widehat{\text{l.o.t.}}=\partial_{A}\Gamma_{jk}^{i}\hat{v}_{S}^{A}+\widehat{\text{l.o.t.}}=:\partial_{\hat{S}}\Gamma_{jk}^{i},
\end{equation}
by equality of the derivatives of the Christoffel's symbols. Using our definitions of the pullback of the linearizations and out notations from \eqref{eq:S=tildeS} and \eqref{eq:S=hatS}, also conclude
\begin{equation} \label{eq:pullback_S}
    F^{*}(\partial_{\hat{S}}\Gamma_{jk}^{i})=F^{*}(\partial_{A}\Gamma_{jk}^{i}\hat{v}_{S}^{A}+\widehat{\text{l.o.t.}})=\partial_{A}\Gamma_{jk}^{i}\tilde{v}_{S}^{A}+\widetilde{\text{l.o.t.}}=\partial_{\tilde{S}}\Gamma_{jk}^{i}, 
\end{equation}
Hence, if we are in the 2-dimensional case,
\begin{align*}
    &\Delta_{g}\tilde{v}_{\mu_{1} \ldots \mu_{N}}^{i}+g^{\alpha \beta}\displaystyle\sum_{\substack{S \in 2^{[N]}\\ 0 \leq |S| \leq N-2}} \partial_{S} \Gamma_{jk}^{i} \sum_{\substack{T \in 2^{[N] \setminus S}\\ |T| \geq 1 \\ T \sqcup S \neq [N]}} \partial_{\alpha}v_{T}^{j} \partial_{\beta}v_{T^{c}}^{k} \\
    &=\Delta_{c^{-1}F^{*}\hat{g}}\tilde{v}_{\mu_{1} \ldots \mu_{N}}^{i}+cF^{*}\hat{g}^{\alpha \beta}\displaystyle\sum_{\substack{S \in 2^{[N]}\\ 0 \leq |S| \leq N-2}} \partial_{S} \Gamma_{jk}^{i} \sum_{\substack{T \in 2^{[N] \setminus S}\\ |T| \geq 1 \\ T \sqcup S \neq [N]}} \partial_{\alpha}v_{T}^{j} \partial_{\beta}v_{T^{c}}^{k} \\
    &=\Delta_{c^{-1}F^{*}\hat{g}}\tilde{v}_{\mu_{1} \ldots \mu_{N}}^{i}+cF^{*}\hat{g}^{\alpha \beta}\displaystyle\sum_{\substack{S \in 2^{[N]}\\ 0 \leq |S| \leq N-2}} \partial_{S} \Gamma_{jk}^{i} \sum_{\substack{T \in 2^{[N] \setminus S}\\ |T| \geq 1 \\ T \sqcup S \neq [N]}} \partial_{\alpha}\tilde{v}_{T}^{j} \partial_{\beta}\tilde{v}_{T^{c}}^{k} \\
    &=\Delta_{c^{-1}F^{*}\hat{g}}\tilde{v}_{\mu_{1} \ldots \mu_{N}}^{i}+cF^{*}\hat{g}^{\alpha \beta}\displaystyle\sum_{\substack{S \in 2^{[N]}\\ 0 \leq |S| \leq N-2}} \partial_{\tilde{S}} \Gamma_{jk}^{i} \sum_{\substack{T \in 2^{[N] \setminus S}\\ |T| \geq 1 \\ T \sqcup S \neq [N]}} \partial_{\alpha}\tilde{v}_{T}^{j} \partial_{\beta}\tilde{v}_{T^{c}}^{k} \\
    &=cF^{*}\left (\Delta_{\hat{g}}\hat{v}_{\mu_{1} \ldots \mu_{N}}^{i}+\hat{g}^{\alpha \beta}\displaystyle\sum_{\substack{S \in 2^{[N]}\\ 0 \leq |S| \leq N-2}} \partial_{\hat{S}} \Gamma_{jk}^{i} \sum_{\substack{T \in 2^{[N] \setminus S}\\ |T| \geq 1 \\ T \sqcup S \neq [N]}} \partial_{\alpha}\hat{v}_{T}^{j} \partial_{\beta}\hat{v}_{T^{c}}^{k} \right) \\
    &=cF^{*}\left (\Delta_{\hat{g}}\hat{v}_{\mu_{1} \ldots \mu_{N}}^{i}+\hat{g}^{\alpha \beta}\displaystyle\sum_{\substack{S \in 2^{[N]}\\ 0 \leq |S| \leq N-2}} \partial_{S} \hat{\Gamma}_{jk}^{i} \sum_{\substack{T \in 2^{[N] \setminus S}\\ |T| \geq 1 \\ T \sqcup S \neq [N]}} \partial_{\alpha}\hat{v}_{T}^{j} \partial_{\beta}\hat{v}_{T^{c}}^{k} \right) \\
    &=0,
\end{align*}
were we used \eqref{eq:rel_metrics_2} on the first equality,  \eqref{eq:rel_N-1_lin} on the second one, \eqref{eq:S=tildeS} on the third equality, \eqref{eq:pullback_S} on the fourth one, \eqref{eq:S=hatS} on the fifth equality, and Lemma \eqref{lemma:linear_eq} (4) on the last one. The same analysis applies in the analytic case, by using \eqref{eq:rel_metrics_an} instead of \eqref{eq:rel_metrics_2}, and similarly, the same analysis applies in the CTA case, by using \eqref{eq:rel_metrics_cta} instead of \eqref{eq:rel_metrics_2} (and is more direct since $F=Id_{M}$). Since $F$ is the identity in the boundary, we obtain \eqref{eq:rel_N_lin} in virtue of the uniqueness of the Dirichlet problem for elliptic PDE (\cite{GT01}*{Theorem 6.15}).
\end{proof} 

\begin{lemma}[General Alessandrini's identity] \label{lemma:general_ale}
    Let $N \geq 3$. Assume that the linearizations up to order $N$ of $u$ and $F^{*}\hat{u}$ coincide, and that the derivatives of Christoffel's symbols of $h$ and $\hat{h}$ at $q$ coincide up to order $N-2$. Then:
    \begin{equation} \label{eq:gen_al}
        \int_{M}v_{N+2}^{\ell}h_{\ell i}g^{\alpha \beta}\partial_{A_{1} \ldots A_{N-1}}(\Gamma_{jk}^{i}-\hat{\Gamma}_{jk}^{i})\sum_{\mathrm{cyc}} \left( \prod_{s=1}^{N-1}v_{s}^{A_{s}}\right)\partial_{\alpha}v_{\mu_{N}}^{j}\partial_{\beta}v_{\mu_{N+1}}^{k}dV_{g}=0,
    \end{equation}
    where the sum is cyclic over $(\mu_{1},\ldots,\mu_{N+1})$, and $h$ and the derivatives of the Christoffel's symbols are evaluated at $q$. In the CTA, since $F=Id_{M}$, the hypothesis means that the linearizations up to order $N$ of $u$ and $\hat{u}$ coincide.
\end{lemma}

\begin{proof}
We have
\begin{align*}
    \int_{M}&v_{N+2}^{\ell}h_{\ell i}(q)g^{\alpha \beta}\displaystyle\sum_{\substack{S \in 2^{[N+1]}\\ 0 \leq |S| \leq N-1}} \partial_{S}\Gamma_{jk}^{i}v_{S}^{A} \sum_{\substack{T \in 2^{[N+1] \setminus S}\\ |T| \geq 1 \\ T \sqcup S \neq [N+1]}} \partial_{\alpha}v_{T}^{j}
    \partial_{\beta}v_{T^{c}}^{k} dV_{g} \\
    =&\int_{M}\hat{v}_{N+2}^{\ell}h_{\ell i}(q)\hat{g}^{\alpha \beta}\displaystyle\sum_{\substack{S \in 2^{[N+1]}\\ 0 \leq |S| \leq N-1}} \partial_{S}\hat{\Gamma}_{jk}^{i}\sum_{\substack{T \in 2^{[N+1] \setminus S}\\ |T| \geq 1 \\ T \sqcup S \neq [N+1]}} \partial_{\alpha}\hat{v}_{T}^{j}
    \partial_{\beta}\hat{v}_{T^{c}}^{k} dV_{\hat{g}} \\
    =&\int_{M}\hat{v}_{N+2}^{\ell}h_{\ell i}(q)\hat{g}^{\alpha \beta}\displaystyle\sum_{\substack{S \in 2^{[N+1]}\\ |S| = N-1}} \partial_{S}\hat{\Gamma}_{jk}^{i}\sum_{\substack{T \in 2^{[N+1] \setminus S}\\ |T| \geq 1 \\ T \sqcup S \neq [N+1]}} \partial_{\alpha}\hat{v}_{T}^{j}
    \partial_{\beta}\hat{v}_{T^{c}}^{k} dV_{\hat{g}} \\
    &+\int_{M}\hat{v}_{N+2}^{\ell}h_{\ell i}(q)\hat{g}^{\alpha \beta}\displaystyle\sum_{\substack{S \in 2^{[N+1]}\\ 0 \leq |S| \leq N-2}} \partial_{\hat{S}}\Gamma_{jk}^{i}\sum_{\substack{T \in 2^{[N+1] \setminus S}\\ |T| \geq 1 \\ T \sqcup S \neq [N+1]}} \partial_{\alpha}\hat{v}_{T}^{j}
    \partial_{\beta}\hat{v}_{T^{c}}^{k} dV_{\hat{g}} \\
    =&\int_{M}\tilde{v}_{N+2}^{\ell}h_{\ell i}(q)F^{*}\hat{g}^{\alpha \beta}\displaystyle\sum_{\substack{S \in 2^{[N+1]}\\ |S| = N-1}} \partial_{\tilde{S}}\hat{\Gamma}_{jk}^{i}\sum_{\substack{T \in 2^{[N+1] \setminus S}\\ |T| \geq 1 \\ T \sqcup S \neq [N+1]}} \partial_{\alpha}\tilde{v}_{T}^{j}
    \partial_{\beta}\tilde{v}_{T^{c}}^{k} F^{*}dV_{\hat{g}} \\
    &+\int_{M}\tilde{v}_{N+2}^{\ell}h_{\ell i}(q)F^{*}\hat{g}^{\alpha \beta}\displaystyle\sum_{\substack{S \in 2^{[N+1]}\\ 0 \leq |S| \leq N-2}} \partial_{\tilde{S}}\Gamma_{jk}^{i}\sum_{\substack{T \in 2^{[N+1] \setminus S}\\ |T| \geq 1 \\ T \sqcup S \neq [N+1]}} \partial_{\alpha}\tilde{v}_{T}^{j}
    \partial_{\beta}\tilde{v}_{T^{c}}^{k} F^{*}dV_{\hat{g}} \\
    =&\int_{M}\tilde{v}_{N+2}^{\ell}h_{\ell i}(q)g^{\alpha \beta}\displaystyle\sum_{\substack{S \in 2^{[N+1]}\\ |S| = N-1}} \partial_{\tilde{S}}\hat{\Gamma}_{jk}^{i}\sum_{\substack{T \in 2^{[N+1] \setminus S}\\ |T| \geq 1 \\ T \sqcup S \neq [N+1]}} \partial_{\alpha}\tilde{v}_{T}^{j}
    \partial_{\beta}\tilde{v}_{T^{c}}^{k} dV_{g} \\
    &+\int_{M}\tilde{v}_{N+2}^{\ell}h_{\ell i}(q)g^{\alpha \beta}\displaystyle\sum_{\substack{S \in 2^{[N+1]}\\ 0 \leq |S| \leq N-2}} \partial_{\tilde{S}}\Gamma_{jk}^{i}\sum_{\substack{T \in 2^{[N+1] \setminus S}\\ |T| \geq 1 \\ T \sqcup S \neq [N+1]}} \partial_{\alpha}\tilde{v}_{T}^{j}
    \partial_{\beta}\tilde{v}_{T^{c}}^{k} dV_{g} \\
    =& \int_{M}v_{N+2}^{\ell}h_{\ell i}(q)g^{\alpha \beta}\displaystyle\sum_{\substack{S \in 2^{[N+1]}\\ |S| = N-1}} \partial_{S}\hat{\Gamma}_{jk}^{i}\sum_{\substack{T \in 2^{[N+1] \setminus S}\\ |T| \geq 1 \\ T \sqcup S \neq [N+1]}} \partial_{\alpha}v_{T}^{j}
    \partial_{\beta}v_{T^{c}}^{k} dV_{g} \\
    &+\int_{M}v_{N+2}^{\ell}h_{\ell i}(q)g^{\alpha \beta}\displaystyle\sum_{\substack{S \in 2^{[N+1]}\\ 0 \leq |S| \leq N-2}} \partial_{S}\Gamma_{jk}^{i}\sum_{\substack{T \in 2^{[N+1] \setminus S}\\ |T| \geq 1 \\ T \sqcup S \neq [N+1]}} \partial_{\alpha}v_{T}^{j}
    \partial_{\beta}v_{T^{c}}^{k} dV_{g},
\end{align*}
where in the first equality we used that the $(N+1)$th derivative of the DN maps coincide, that $h(q)=\hat{h}(q)$, and Lemma \ref{lemma:N_int_id} for $N+1$; in the second we separated the sum, then that the derivatives of Christoffel's symbols agree up to order $N-2$, and the notation from \eqref{eq:S=hatS}; in the third equality we used the notation from \eqref{eq:S=tildeS}; in the fourth we used \eqref{eq:rel_metrics_2} together with \eqref{eq:rel_vols_2} if $\dim M=2$, and \eqref{eq:rel_metrics_an} together with \eqref{eq:rel_vols_n} in the higher dimensional and analytic case; and finally in the last equality we used that the linearizations of $u$ and $F^{*}\hat{u}$ up to order $N$ coincide. In the CTA case, the same proof works taking $F=Id_{M}$. If we cancel the repeated terms, which are the ones on the sum with $0 \leq |S| \leq N-2$, we obtain
\[ \int_{M}v_{N+2}^{\ell}h_{\ell i}(q)g^{\alpha \beta}\displaystyle\sum_{\substack{S \in 2^{[N+1]}\\ |S|= N-1}} (\partial_{A}\Gamma_{jk}^{i}-\partial_{A}\hat{\Gamma}_{jk}^{i})v_{S}^{A} \sum_{\substack{T \in 2^{[N+1] \setminus S}\\ |T| \geq 1 \\ T \sqcup S \neq [N+1]}} \partial_{\alpha}v_{T}^{j}
    \partial_{\beta}v_{T^{c}}^{k} dV_{g}=0, \]
which is equivalent to \eqref{eq:gen_al}.
\end{proof}

\begin{lemma} \label{lemma:ind_der_chris}
Let $N \geq 2$. If the linearizations up to order $N-1$ of $u$ and $F^{*}\hat{u}$ coincide, and the derivatives of Christoffel's symbols of $h$ and $\hat{h}$ at $q$ coincide up to order $N-2$, then the derivatives of Christoffel's symbols of $h$ and $\hat{h}$ at $q$ coincide up to order $N-1$. Again, the first hypothesis in the CTA means that the linearizations up to order $N-1$ of $u$ and $\hat{u}$ are the same.
\end{lemma}

\begin{proof}
We already proved this for $N=2$ in Section \ref{sub:second_der}. Hence, take $N \geq 3$. By Lemma \ref{lemma:rel_N_lin}, we have that the linearizations up to order $N$ are equal. Hence, Lemma \ref{lemma:general_ale} is valid. Fix $A_{1}^{0},\ldots,A_{N-1}^{0},j_{0},k_{0},\ell_{0} \in \{1,\ldots,n\}$. Let $a \in C^{\infty}(\partial M)$ to be non-constant, let $C$ be a positive constant. Choose 
\[ f_{\mu_{s}}^{A_{s}}=\begin{cases}
    C & \text{if } A_{s}=A_{s}^{0}, \\
    0 & \text{in other case},
\end{cases} \quad f_{N}^{j}=\begin{cases}
    a & \text{if } j=j_{0}, \\
    0 & \text{in other case},
\end{cases}\]
\[ f_{N+1}^{k}=\begin{cases}
    a & \text{if } k=k_{0}, \\
    0 & \text{in other case},
\end{cases} \quad f_{N+2}^{\ell}=\begin{cases}
    C & \text{if } \ell=\ell_{0}, \\
    0 & \text{in other case},
\end{cases} \]
where $s \in \{1,\ldots,N-1\}$. Then, by uniqueness of solutions to elliptic Dirichlet problems (\cite{GT01}*{Theorem 6.15}), we obtain that
\[ v_{\mu_{s}}^{A_{s}}=v_{\mu_{N}}^{j}=v_{\mu_{N+1}}^{k}=v_{N+2}^{\ell}=0,\]
$A_{s} \in \{1,\ldots,n \} \setminus \{A_{s}^{0}\}$ for $s \in \{1,\ldots,N-1\}$, $j \in \{1,\ldots,n \} \setminus \{j_{0}\}$, $k \in \{1,\ldots,n \} \setminus \{k_{0}\}$, and $\ell \in \{1,\ldots,n \} \setminus \{\ell_{0}\}$, respectively, and
\[ v_{N+2}^{\ell_{0}}=v_{\mu_{s}}^{A_{s}^{0}}=C, \qquad v_{N}^{j_{0}}=v_{N+1}^{k_{0}}. \]
for $s \in \{1,\ldots,N-1\}$. Then, Lemma \ref{lemma:general_ale} implies
\[ C^{N}h_{\ell_{0}i}(q)\partial_{A_{1}^{0} \ldots A_{N-1}^{0}}(\Gamma_{j_{0}k_{0}}^{i}(q)-\hat{\Gamma}_{j_{0}k_{0}}^{i}(q))\int_{M}|dv_{N}^{j_{0}}|_{g}^{2}dV_{g}=0. \]
The fact that $a$ is non-constant gives
\[ h_{\ell_{0}i}(q)\partial_{A_{1}^{0} \ldots A_{N-1}^{0}}(\Gamma_{j_{0}k_{0}}^{i}(q)-\hat{\Gamma}_{j_{0}k_{0}}^{i}(q))=0. \]
Since we can repeat the same proof with any $\ell_{0}$, and $h$ is invertible, we conclude that 
\[ \partial_{A_{1}^{0} \ldots A_{N-1}^{0}}\Gamma_{j_{0}k_{0}}^{i}(q)=\partial_{A_{1}^{0} \ldots A_{N-1}^{0}}\hat{\Gamma}_{j_{0}k_{0}}^{i}(q). \]
Finally, as $A_{1}^{0},\ldots,A_{N-1}^{0},j_{0},k_{0}$ were arbitrary, we conclude that the derivatives of the Christoffel's symbols are the same.    
\end{proof}

\begin{prop} \label{prop:jets}
In the setting of Theorem \ref{thm:main}, the jets of $h$ and $\hat
h$ coincide at $q$.
\end{prop}

\begin{proof}
Consider  an initial condition of the form \eqref{eq:boundary_value} with $N_{0}+1=N+2$. By hypothesis and propositions \ref{prop:first_der} and \ref{prop:second_der}, we already have that the derivatives up to order $2$ at $q$ of the metrics coincide. We will focus on the higher derivatives. By lemmas \ref{lemma:same_chris} and \ref{lemma:same_chris_der1}, we have that the derivatives up to order $1$ at $q$ of the Christoffel's symbols are the same. Hence, in light of Lemma \ref{lemma:ind_der}, to obtain the equality for the third derivative of the metrics, we need to show that the second derivatives of the Christoffel's symbols are same at $q$. However, this is implied by Lemma \ref{lemma:ind_der_chris}. Note that we can apply Lemma \ref{lemma:rel_N_lin} and obtain equality of the linearizations up to order $3$. Then, we can use again Lemma \ref{lemma:ind_der_chris} to obtain that derivative up to order $3$ the Christoffel's symbols are the same at $q$. Then Lemma \ref{lemma:ind_der} shows that the we have that the derivatives of order $4$ at $q$ are the same. Inductively, we obtain that the derivatives of the Christoffel's symbols up to order $N-1$ at $q$, and of the metrics up to order $N-1$ at $q$ are the same. Hence, Lemma \ref{lemma:ind_der} implies that we have equality of all the derivatives of the metrics up to order $N$. To obtain the equality of the equality of the derivatives of the metric of order $N+1$ at $q$, we repeat the exact same proof, but we now consider \eqref{eq:boundary_value} with $N+3$. In this way, we are able to show that all the derivatives of the metric at $q$ are equal.    
\end{proof}

\begin{proof}[Proof of Theorem \ref{thm:main}]
    It follows from Lemma \ref{lemma:metrics_domain} and Proposition \ref{prop:jets}.
\end{proof}

\section{Energy Rigidity} \label{sec:er}

Finally, we deal with Corollary \ref{cor:er}. This result will be a consequence of the following lemma:

\begin{lemma} \label{lemma:e_to_dn}
Consider the setting of Corollary \ref{cor:er}. Let $u_{q,f}$ the harmonic map with boundary value $q+f$, where $q+f \in B_{r}(q)$. Then, the boundary $\partial M$, the metrics $g|_{\partial M}$, $h|_{B_{r}(q)}$, and the Dirichlet energies $E(u_{q+f})$ for all $f$ with small norm, determine $\Lambda_{g,h}(q+f)$ for all such $f$. Furthermore, Fréchet derivatives of the map $q+f \mapsto E(u_{q+f})$ of the order $k+1$ at $f=0$ determine the Fréchet derivatives of the map $q+f \mapsto \Lambda_{g,h}(q+f)$ at $f=0$ to the order $k$.
\end{lemma}

\begin{proof}
Take $\varphi \in C^{\infty}(\partial M,\R^{n})$. Let $u=u_{q+f}$ and $u_{t}=u_{q+f+t\varphi}$ be the harmonic maps with boundary value given by the lower indices and homotopic to them. Denote by $W$ the variational field of $u_{t}$, that is, $W=\frac{d}{dt}|_{t=0}u_{t}$. Then,
\begin{align*}
    \frac{d}{dt}\bigg|_{t=0}E(u_{t})=&\frac{d}{dt}\bigg|_{t=0}\frac{1}{2}\int_{M} g^{\alpha \beta}h_{ij}(u)\partial_{\alpha}u_{t}^{i}\partial_{\beta}u_{t}^{j}dV_{g} \\
    =& \frac{1}{2} \int_{M} g^{\alpha \beta} \partial_{k}h_{ij}(u) W^{k}\partial_{\alpha}u^{i}\partial_{\beta}u^{j}dV_{g} \\
    &+\int_{M}g^{\alpha \beta}h_{ij}(u)\partial_{\alpha}u^{i}\partial_{\beta}W^{j}dV_{g} \\
    =& \frac{1}{2} \int_{M} g^{\alpha \beta} \partial_{k}h_{ij}(u) W^{k}\partial_{\alpha}u^{i}\partial_{\beta}u^{j}dV_{g} \\
    &-\int_{M}(h_{ij}\Delta_{g}u^{i}+g^{\alpha \beta}\partial_{k}h_{ij}(u)\partial_{\alpha}u^{i}\partial_{\beta}u^{k})W^{j}dV_{g} \\
    &+\int_{\partial M} W^{j}h_{ij}(u) \partial_{\eta}u^{i} dS_{g}, \\
    =&\int_{\partial M} W^{j}h_{ij}(u) \partial_{\eta}u^{i} dS_{g},
\end{align*}
where in the third equality we used that integration by parts, and in the fourth we used that $u_{q,f}$ is an harmonic map. In addition we have $W|_{\partial M}=\varphi$. Therefore, we conclude 
\begin{equation} \label{eq:variation}
    \frac{d}{dt}\bigg|_{t=0}E(u)=\int_{\partial M} \varphi^{j}h_{ij}(q+f) \partial_{\eta}u_{q,f}^{i} dS_{g}.
\end{equation}
Since $\partial M$, $g_{\partial M}$, and $h|_{B_{r}(q)}$ are known, and $\varphi$ is arbitrary, we conclude from \eqref{eq:variation} that $\frac{d}{dt}|_{t=0}E(u)$ (with varying values of $\varphi$) determines $\Lambda_{g,h}(q+f)$. 

Regarding higher derivatives, by taking the derivative to 
\begin{equation} \label{eq:dn_fre}
    h_{ij}(q+f)\partial_{\eta}u(q+f)
\end{equation}
we can determine the first derivative of the DN map. Using this new information and taking the second derivative in \eqref{eq:dn_fre}, we can determine the second derivative of the DN map, and so on.
\end{proof}

\begin{rmk}
Note that on the proof of Theorem \ref{thm:main}, we only used the linearizations of the DN for harmonic maps and we obtained the jet of $h$ at $q$, and from equations \eqref{eq:variation}, \eqref{eq:dn_fre}, we see that these linearizations depend on $\partial M$, $g|_{\partial M}$, and the jet of $h$ at $q$ only. Hence, assuming the knowledge of $\partial M$ and $g|_{\partial M}$, we see that knowing the linearizations of $\Lambda_{g,h}$ and the jet of $h$ at $q$ determine each other.
\end{rmk}

\end{document}